\documentclass[11pt]{amsart}
\usepackage{amsmath,amssymb,amscd,euscript,mathrsfs}

\newtheorem{theorem}{Theorem}[section]
\newtheorem{lemma}[theorem]{Lemma}
\newtheorem{corollary}[theorem]{Corollary}
\newtheorem{proposition}[theorem]{Proposition}

\theoremstyle{definition}
\newtheorem{definition}[theorem]{Definition}
\newtheorem{remark}[theorem]{Remark}

\renewcommand{\phi}{\varphi}
\renewcommand{\tilde}{\widetilde}
\newcommand{\OOmega}{\mathbf{\Omega}}
\newcommand{\Fl}{\mathcal{F}}

\newcommand{\Hess}{\mathcal{H}}
\newcommand{\HHess}{\mathscr{H}}
\newcommand{\Sch}{\mathfrak{S}}

\newcommand{\Z}{\mathbb{Z}}
\newcommand{\C}{\mathbb{C}}

\DeclareMathOperator{\rk}{rk}
\DeclareMathOperator{\Hom}{Hom}

\newcommand{\isom}{\cong}
\renewcommand{\setminus}{\smallsetminus}

\newcommand{\define}{\textbf}

\begin{document}

\title{Schubert polynomials and classes of Hessenberg varieties}
\author{Dave Anderson}
\address{Department of Mathematics\\University of Michigan\\Ann Arbor, MI 48109}
\email{dandersn@umich.edu}
\author{Julianna Tymoczko}
\address{Department of Mathematics, University of Iowa, Iowa City, IA 52242}
\email{tymoczko@math.iowa.edu}
\thanks{DA was partially supported by an RTG fellowship, NSF Grant 0502170.  JT was partially supported by an NSF postdoctoral fellowship, NSF Grant 0402874.}

\date{October 2, 2008}

\begin{abstract}
Regular semisimple Hessenberg varieties are a family of subvarieties of the flag variety that arise in number theory, numerical analysis, representation theory, algebraic geometry, and combinatorics.  We give a ``Giambelli formula'' expressing the classes of regular semisimple Hessenberg varieties in terms of Chern classes.  In fact, we show that the cohomology class of each regular semisimple Hessenberg variety is the specialization of a certain double Schubert polynomial, giving a natural geometric interpretation to such specializations.  We also decompose such classes in terms of the Schubert basis for the cohomology ring of the flag variety.  The coefficients obtained are nonnegative, and we give closed combinatorial formulas for the coefficients in many cases.  We introduce a closely related family of schemes called regular nilpotent Hessenberg schemes, and use our results to determine when such schemes are reduced.
\end{abstract}

\maketitle

\section{Introduction}

A fundamental open problem in combinatorial representation theory and algebraic geometry is to give closed formulas for calculations in the cohomology ring of the full flag variety $\Fl = SL_n/B$ in terms of the basis of Schubert classes.  This paper answers a part of this question by expressing the classes of certain subvarieties of the flag variety called regular semisimple Hessenberg varieties in terms of the Schubert classes.  We give closed combinatorial formulas in many cases for the coefficients of the Schubert classes in this expansion and in particular show that the coefficients are nonnegative.  We prove that the cohomology class of each regular semisimple Hessenberg variety is the specialization of a certain double Schubert polynomial, thus giving a natural geometric interpretation to such specializations.  As an application, we determine when a family of Hessenberg schemes called regular {\em nilpotent} Hessenberg schemes are reduced.  The key observation is that Hessenberg varieties can be realized as certain degeneracy loci in the flag variety.

The cohomology ring of the flag variety can be expressed as a quotient $H^*(\Fl) = \mathbb{Z}[x_1,x_2,\ldots,x_n]/I$, where $I$ is the ideal generated by elementary symmetric polynomials $e_1(x),\ldots,e_n(x)$; the equivariant cohomology can similarly be expressed as $H^*(\Fl) = \mathbb{Z}[x_1,\ldots,x_n;y_1,\ldots,y_n]/J$ where $J$ is generated by differences $e_i(x)-e_i(y)$.  {\it Double Schubert polynomials} can be interpreted as special polynomial representatives for the classes of Schubert varieties in $H^*_T(\Fl)$.  In this paper, we study the double Schubert polynomials corresponding to {\em dominant} permutations (see Section \ref{sec:perms}).  Each dominant double Schubert polynomial is a product of binomials, as shown in \cite[(6.14)]{mac}.  Proposition \ref{prop:not monomial} proves the converse: dominant double Schubert polynomials are the only Schubert polynomials that are a product of binomials.

In combinatorics, specializing polynomials at particular values of the variables is a classical way to find new identities.  Certain specializations of Schubert polynomials have been studied before; see e.g. \cite{br}, \cite{bjs}, \cite{llt}, \cite{p}, and \cite{KT}.  We specialize double Schubert polynomials at the permutation $w_0$ defined by $w_0(i)=n-i+1$ for all $i$, which means imposing the conditions $y_i=x_{w_0(i)}$ for each $i$.  The formulas obtained are essentially the same as those for localized Schubert classes in equivariant cohomology, but the meaning is quite different.  In equivariant cohomology, one replaces $x$'s with $y$'s to obtain the torus weight at a fixed point; here we do the opposite, replacing $y$'s with $x$'s to obtain a formula in ordinary cohomology.  Proposition \ref{prop:hessclassform} gives geometric meaning to this operation by showing that the specialization of a dominant double Schubert polynomial at the longest permutation $w_0$ is the class of a variety called a {\em regular semisimple Hessenberg variety} inside $H^*(\Fl)$.  As a corollary of the geometric interpretation, we see that when such a specialized double Schubert polynomial is written in terms of the basis of (single) Schubert polynomials, the coefficients corresponding to permutations $u\in S_n$ are nonnegative (see Proposition \ref{prop:positive}).  This joins the ranks of other results in which positivity is proved through geometric arguments rather than combinatorial ones, e.g. \cite{G}, \cite{BFR}, \cite{KMS}, or \cite{Paddendum}.  It also motivates a combinatorial problem: {\it Express the class of a double Schubert polynomial specialized at a permutation in terms of single Schubert polynomials}.

{\em Hessenberg varieties} are subvarieties of the flag variety that were defined in \cite{mps}.  They depend on two parameters: a function $h:\{1,\ldots,n\} \to \{1,\ldots,n\}$ (called a \emph{Hessenberg function}), and a linear map $X:\C^n\to\C^n$.  (We will give the precise definitions in Section \ref{sec:hessenberg}.)  Examples appear in representation theory \cite{bm}, in the Langlands program \cite{gkm}, and in numerical analysis \cite{mps}, among other applications.  In Section \ref{sec:hessenberg}, we give a new construction of Hessenberg varieties as certain degeneracy loci, and define the more general \emph{Hessenberg schemes}.

The discrete parameter $h$ can also be encoded as a partition or a (dominant) permutation (see Section \ref{sec:hessenberg}).  We will usually work with a genericity assumption on the continuous parameter $X$: namely, that it is a regular semisimple endomorphism.  In this case, the corresponding Hessenberg variety is called {\em regular semisimple} and is a smooth, pure-dimensional projective algebraic variety \cite{mps}.  Regular semisimple Hessenberg varieties themselves have been used to produce new combinatorial identities \cite{fu}, \cite{i} and permutation representations \cite{s}. 

Our main results explicitly compute the classes of regular semisimple Hessenberg varieties inside $H^*(\Fl(n))$ in terms of the basis of Schubert classes, giving a partial solution to the specialization problem mentioned above.  Let $\Omega_v$ denote the Schubert variety corresponding to a permutation $v$.  We give a complete and explicit formula for the coefficient of $[\Omega_v]$ when the Hessenberg class is in {\em the stable range}, which occurs when the codimension of the Hessenberg variety is sufficiently small.

\medskip

\noindent {\bf Theorem \ref{thm:stableclasses}.} {\em Let $w$ be the dominant permutation defining a regular semisimple Hessenberg variety $\Hess_w$, and suppose $w \in S_k \subseteq S_n$ for $2k \leq n$.  Then
\begin{eqnarray*}
{[\Hess_{w}] = \Sch_w(x_1,\ldots,x_n;x_n,\ldots,x_1) = \sum_{v^{-1}u=w} {[\Omega_{u\,w_0\,v\,w_0}]}}
\end{eqnarray*}
in $H^*(\Fl(n))$, where the sum is over permutations $u,v$ such that $v^{-1}u=w$ and $\ell(u)+\ell(v)=\ell(w)$.}

\medskip

When the Hessenberg variety has higher codimension, we focus on the ones associated to the permutation $w_k = \prod_{i=1}^k s_{n-i}$, for simple transpositions $s_i = (i,i+1)$.  (This is the cycle $[1,2,\ldots,n-k-1,n,n-k,n-k+1,\ldots,n-1]$.)  These permutations appear in Pieri formulas for classical Schubert calculus, namely closed formulas for a product of the Schubert polynomial for $w_k$ with an arbitrary Schubert polynomial.  Pieri formulas are key to giving an explicit presentation of the ring $H^*(\Fl)$ \cite{FLgrothpieri}, \cite{Pquantpieri}, \cite{Reqvtpieri}, \cite{Soflagspieri}.

Our formula is given in terms of certain paths in the Bruhat graph of the permutation group.  A {\em restricted path} is a sequence of permutations $u_1, u_2, \ldots, u_p$ such that for each $i=1, \ldots, p-1$, the length of $u_{i+1}$ is $\ell(u_{i})-1$ and $u_{i+1}$ can be written as the product $u_i s_{1j}$ for some transposition $s_{1j}=(1j)$.

\medskip

\noindent {\bf Theorem \ref{thm:rows}.} {\em If $u$ is a permutation of length $k$ and $\Hess$ is the regular semisimple Hessenberg variety associated to $w_k$, then the coefficient of the Schubert class corresponding to $u$ in the class of $\Hess$ in $H^*(\Fl(n))$ is the number of restricted paths from $u$ to a permutation $s_{n-p}s_{n-p-1} \cdots s_{n-k}$, where $p$ is the number of terms in the path.}

\medskip

If $X$ is a regular nilpotent endomorphism of $\C^n$, the corresponding Hessenberg varieties are called regular nilpotent; these have been used to compute the quantum cohomology of the flag variety \cite{Ko}, \cite{R} and they are related to hyperplane arrangements \cite{ST}.  We also consider regular nilpotent Hessenberg \emph{schemes}, which may have additional nonreduced structure.  As discussed in Section \ref{limits}, regular nilpotent Hessenberg schemes are flat limits of regular semisimple Hessenberg varieties.  It follows that the decomposition into Schubert classes of the class of a regular nilpotent Hessenberg scheme is the same as the decomposition into Schubert classes for regular semisimple Hessenberg varieties.  We use this to determine when the regular nilpotent Hessenberg scheme associated to $w_k$ is reduced.  

\medskip

\noindent {\bf Theorem \ref{thm:nilp scheme}.}  {\em The regular nilpotent Hessenberg scheme associated to $w_k$ is reduced if $k < n-1$.  If $k=n-1$ then the regular nilpotent Hessenberg variety $\Hess(N, w_k)$ and the regular semisimple Hessenberg variety $\Hess(S, w_k)$ satisfy 
\[ [\Hess(S,w_k)] = n [\Hess(N,w_k)].\]}

\medskip

We fix notation, conventions, and basic definitions in Section \ref{sec:basic}.  Section \ref{sec:hessenberg} contains the definitions of Hessenberg varieties and schemes, as well as the basic propositions giving Chern class formulas for their cohomology classes.  In Section \ref{sec:perms}, we discuss dominant permutations and their Schubert polynomials.  Sections \ref{sec:computing} through \ref{limits} are devoted to the proofs of the results discussed above.  Finally, in a short appendix we prove a connectivity result for Hessenberg varieties; this can be read independently.

\medskip
\noindent
{\it Acknowledgements.}  We thank Alex Yong for an inspiring conversation, Bill Fulton for helpful comments on the manuscript, and Ezra Miller for simplifying our proof of Proposition \ref{prop:not monomial}.

\section{Basic setup} \label{sec:basic}

\subsection{Flag varieties}

Let $\Fl = \Fl(n)$ be the variety of complete flags in an $n$-dimensional vector space $E$.  On $\Fl$, there is the universal sequence
\begin{eqnarray*}
S_1 \subset S_2 \subset \cdots \subset S_n = E_{\Fl} = Q_n \to Q_{n-1} \to \cdots \to Q_1,
\end{eqnarray*}
where $S_{\bullet}$ is the tautological flag of subbundles and $Q_{\bullet}$ is the universal flag of quotients.  (Thus $Q_i = E/S_{n-i}$.)  Let $x_i = c_1(\ker(Q_i \to Q_{i-1}))$.  Then the cohomology ring of $\Fl$ can be presented as
\begin{eqnarray*}
H^*(\Fl) = \Z{[x_1,\ldots,x_n]}/I,
\end{eqnarray*}
where $I$ is the ideal generated by nonconstant symmetric polynomials in $x_1,\ldots,x_n$.

The \define{Schubert varieties} in $\Fl$ give a geometric basis for $H^*(\Fl)$.  To define them, we set some notation.  For any permutation $w\in S_n$, define the $(q,p)$-rank of $w$ to be
\begin{eqnarray*}
r_w(q,p) = \#\{ i \leq q \,|\, w(i) \leq p \}.
\end{eqnarray*}

If we fix a flag $E_{\bullet}$, the corresponding Schubert varieties are
\begin{equation} \label{eqn:schdefn}
\Omega_w(E_{\bullet}) = \{ F_{\bullet} \in \Fl \,|\, \dim(E_p \cap F_q) \geq r_{w\,w_0}(q,p) \text{ for } 1\leq p,q\leq n \}.
\end{equation}
(Here $w_0$ is the longest permutation in $S_n$, sending $i$ to $n+1-i$.  Our conventions agree with \cite{flags} but differ from \cite{yt}.)  The \emph{Schubert cell} $\Omega^o_w(E_{\bullet})$ is the set of flags in $\Omega_w$ with equality in \eqref{eqn:schdefn}.   We have $\Omega_w(E_{\bullet}) = \overline{\Omega^o_w(E_{\bullet})}$.

\subsection{Permutations and Bruhat order}

We use $s_i$ to denote the simple transposition interchanging $i$ and $i+1$, and we write $e$ for the identity.  We also use \emph{one-line notation} for permutations, writing $w = [w(1),w(2),\ldots,w(n)]$, and sometimes omitting the brackets.  Thus $s_1\,s_2 = [2,3,1] = 2\;3\;1$ in $S_3$.

There are several equivalent descriptions of the \define{Bruhat order} on $S_n$; we refer to \cite[\S1]{mac} or \cite[\S2]{bb} for the following facts.  Say $w$ covers $v$ if there is a reduced expression 
\[ w = s_{i_1} s_{i_2} \cdots s_{i_k},\]
and some $j$ such that
\[ v = s_{i_1} \cdots \widehat{s_{i_j}}\cdots s_{i_k} \]
is a reduced expression for $v$.  Then the Bruhat order is the partial order generated by this covering relation.

We also use two characterizations of the \define{length} of a permutation.  If $w = s_{i_1} \cdots s_{i_\ell}$ is a reduced expression, then the length of $w$ is $\ell(w) = \ell$.  The length can also be described as the number of {\em inversions} of $w$, namely $\ell(w) = \# \{ i< j \,|\, w(i) > w(j) \}$.  The Schubert variety $\Omega_w$ has codimension $\ell(w)$ in $\Fl$.  (This can be taken as a geometric definition of the length of $w$.)

\subsection{Schubert polynomials}

Define \define{divided difference operators} $\partial_i$ on \linebreak $\Z{[x_1,\ldots,x_n; y_1, \ldots, y_n]}$ by the formula
\begin{eqnarray*}
\partial_i P(x;y) = \frac{ P(x;y) - P(\ldots,x_{i+1},x_i,\ldots;y) }{x_i - x_{i+1}}.
\end{eqnarray*}
If $w = s_{i_1}\, s_{i_2} \cdots s_{i_\ell}$ is a reduced expression, set $\partial_w = \partial_{i_1} \circ \partial_{i_2} \circ \cdots \circ \partial_{i_\ell}$.  Since the divided difference operators satisfy the relations $\partial_i \partial_{i+1} \partial_i = \partial_{i+1} \partial_i \partial_{i+1}$ and $\partial_i^2 = 0$, the operator $\partial_w$ is independent of the choice of word for $w$.

The \define{double Schubert polynomials} of Lascoux and Sch\"utzenberger are defined as follows.  For $w_0 = {[n,n-1,\ldots,1]} \in S_n$ in one-line notation, set
\begin{eqnarray*}
\Sch_{w_0}(x;y) = \prod_{i+j\leq n} (x_i - y_j).
\end{eqnarray*}
Now for any $w\in S_n$, write $w = w_0\, u$ with $\ell(u) = \ell(w_0) - \ell(w)$, and set
\begin{eqnarray*}
\Sch_w(x;y) = \partial_u \Sch_{w_0}(x;y).
\end{eqnarray*}
Equivalently, $\Sch_w$ is defined inductively by
\begin{eqnarray*}
\partial_i \Sch_w(x;y) = \left\{ \begin{array}{cl} \Sch_{w\, s_i}(x;y) & \text{if } \ell(w) >\ell(w\, s_i) ; \\ 0 &\text{if } \ell(w) < \ell(w\, s_i). \end{array} \right.
\end{eqnarray*}

Setting the $y$ variables to $0$, one obtains the \define{single Schubert polynomials} $\Sch_w(x) = \Sch_w(x;0)$.  Interpreting the $x$'s as Chern classes of the universal bundles on $\Fl$, we have ${[\Omega_w]} = \Sch_w(x)$ in $H^*\Fl$.

An important fact about (single) Schubert polynomials is that they form a basis for all polynomials, as $w$ ranges over $S_\infty = \bigcup S_n$.  More precisely, $\{\Sch_w(x) \,|\, w(i) < w(i+1) \text{ for all } i \geq n\}$ is an additive basis for $\Z{[x_1,\ldots,x_n]}$ (see \cite[Proposition 10.6]{yt}).

\subsection{Degeneracy loci}\label{subsec:degeneracy}

In the most general situation we consider, $V$ is a rank $n$ vector bundle on a scheme $Z$, with two complete flags
\[ U_1 \subset U_2 \subset \cdots \subset U_n = V = W_n \to W_{n-1} \to \cdots \to W_1. \]
Given an endomorphism $X:V\to V$ and a permutation $w$, there is a \define{degeneracy locus}:
\begin{eqnarray*}
\OOmega_w(X) = \{x \in Z \,|\, \rk(U_p(x) \to V(x) \xrightarrow{X} V(x) \to W_q(x)) \leq r_w(q,p) \text{ for } 1 \leq p,q \leq n \}. 
\end{eqnarray*}
As a scheme, $\OOmega_w(X)$ is the inverse image of a universal degeneracy locus; see \cite{flags} for details.

The essential fact we need is a special case of \cite[Theorem 8.2]{flags}: If $Z$ is smooth, and if $\OOmega_w(X)$ has expected codimension $\ell(w)$, then $\OOmega_w(X)$ is Cohen--Macaulay, and the cohomology class of $\OOmega_w(X)$ in $H^*Z$ is represented by a \emph{double Schubert polynomial}:
\begin{eqnarray} \label{eqn:deglocform}
{[\OOmega_w(X)]} = \Sch_w(x_1,\ldots,x_n;y_1,\ldots,y_n) ,
\end{eqnarray}
where $x_i = c_1(\ker(W_i \to W_{i-1}))$ and $y_i = c_1(U_i/U_{i-1})$.  This formula holds for $\OOmega_w(X)$ considered as a scheme---in particular, if $\OOmega_w(X)$ has nonreduced components, there will be multiplicities in the expression for ${[\OOmega_w(X)]}$.

The paradigmatic case is that of Schubert varieties in $\Fl(n)$.  If $E_1 \subset \cdots \subset E_n = V$ is a fixed flag of trivial subbundles, and $Q_{\bullet}$ is the universal flag of quotient bundles on $\Fl$, then the Schubert variety $\Omega_w(E_{\bullet})$ is the degeneracy locus $\OOmega_w(id)$ for the sequence $E_1 \subset \cdots \subset E_n = V = Q_n \to \cdots \to Q_1$.

\section{Hessenberg schemes and Hessenberg coefficients} \label{sec:hessenberg}

As defined in \cite{mps} or \cite{t1}, the \define{Hessenberg varieties} in $\Fl(n)$ depend on a linear map $X:\C^n\to\C^n$ and a \define{Hessenberg function} $h$, and are given by
\begin{eqnarray}
\Hess(X,h) = \{ F_{\bullet} \in \Fl \,|\, X(F_i) \subseteq F_{h(i)} \text{ for all } i \}.
\end{eqnarray}
The function $h:{[n]} \to {[n]}$ is required to be nondecreasing and to satisfy $h(i) \geq i$ for all $i$.

More generally, we consider \define{Hessenberg schemes}, which are certain degeneracy loci in a flag bundle.  Indeed, the condition $X(F_i) \subseteq F_{h(i)}$ is the same as requiring $\rk(S_i(x) \to V \xrightarrow{X} V \to Q_{n-h(i)}(x)) = 0$, since $Q_{n-h(i)} = V/S_{h(i)}$.  There is a unique $w\in S_n$ whose rank matrix $r_w$ determines exactly these conditions; it is given by $r_w(n-h(i),i) = 0$ for each $i$ together with all implied conditions.  This data is also encoded in the partition $(n-h(1),n-h(2),\ldots,n-h(n))$, or equivalently, its conjugate\footnote{The \emph{conjugate} of a partition is the partition with rows and columns transposed; formally, $\mu'_i = \#\{ j\,|\, \mu_j\geq i\}$.} partition $\lambda$.  We write $w=w(h)=w(\lambda)$.  (Section \ref{sec:perms} has an explicit formula.)  For example, if $n=4$, and $h(1)=h(2)=3$, $h(3)=h(4)=4$, then $\lambda = (2)$ and $w(\lambda)$ is the permutation with rank matrix given by 
\begin{eqnarray*}
\begin{array}{|c|c|c|c|} \hline
0 & 0 & 1 & 1 \\ \hline
1 & 1 & 2 & 2 \\ \hline
1 & 2 & 3 & 3 \\ \hline
1 & 2 & 3 & 4 \\ \hline \end{array}\,.
\end{eqnarray*}
(The entry in the $(q,p)$ position is $r_w(q,p)$.)  Thus $w = {[3,1,2,4]}$.  

\begin{definition}
Let $V$ be a vector bundle of rank $n$ on a scheme $Z$, with a complete flag of subbundles $U_\bullet$; put $W_i = V/U_{n-i}$.  For an endomorphism $X:V\to V$ and a Hessenberg function $h$, the \define{Hessenberg scheme} is the degeneracy locus
\begin{eqnarray*}
\HHess(X,h) = \OOmega_{w(h)}(X) \subseteq Z.
\end{eqnarray*}
The \define{expected codimension} of $\HHess(X,h)$ is the length of the permutation $w(h)$; that is, the codimension of the Schubert variety $\Omega_{w(h)}$ in $\Fl$.  It is equal to $\sum_{i=1}^n (n-h(i))$.
\end{definition}

Thus when $Z=\Fl(n)$ and $U_\bullet = S_\bullet$ is the tautological flag, the Hessenberg variety is just the reduced structure on the Hessenberg scheme: $\Hess(X,h) = \HHess(X,h)_{\mathrm{red}}$.

\begin{proposition} \label{prop:hessclassform}
If $Z$ is smooth and $\HHess(X,h)$ has expected codimension,
\begin{eqnarray} \label{eqn:hessclassform}
[\HHess(X,h)] = \Sch_{w(h)}(x_1,\ldots,x_n;x_n,\ldots,x_1)
\end{eqnarray}
in $H^*Z$, where $x_i = c_1(\ker(W_i \to W_{i-1}))$.
\end{proposition}

\begin{proof}
These hypotheses together with the degeneracy locus formula (\ref{eqn:deglocform}) say \linebreak ${[\HHess(X,h)]} = \Sch_{w(h)}(x;y)$.  Since $\ker(W_i \to W_{i-1}) = U_{n+1-i}/U_{n-i}$, the $x$ and $y$ variables are related by $x_i = y_{n+1-i}$.
\end{proof}

Two cases are of particular interest.  In this article we focus on the first:
\begin{corollary}
Suppose $\HHess(X,h) \subseteq \Fl(n)$ has expected codimension and is reduced.  Then
\renewcommand{\theenumi}{\alph{enumi}}
\begin{enumerate}
\item $[\Hess(X,h)] = \Sch_{w(h)}(x_1,\ldots,x_n;x_n,\ldots,x_1)$ in $H^*\Fl(n)$, where $x_i = c_1(\ker(Q_i \to Q_{i-1}))$; and

\medskip
\item $[\Hess(X,h)]^T = \Sch_{w(h)}(x_1,\ldots,x_n;x_n,\ldots,x_1)$ in $H_T^*\Fl(n)$, where $x_i = c^T_1(\ker(Q_i \to Q_{i-1}))$ and $T\subset SL_n$ is a torus commuting with $X$.
\end{enumerate}
\end{corollary}

\begin{proof}
The first statement is immediate, since $\HHess(X,h) = \Hess(X,h)$.  For the second, apply Proposition \ref{prop:hessclassform} with $Z = ET \times^T \Fl(n)$.
\end{proof}

\begin{remark}
The variables corresponding to torus weights do not appear explicitly in our formula for $[\Hess(X,h)]^T$ --- the equivariant part of the formula lies in the definition of $x_i$ as an equivariant Chern class.
\end{remark}

\begin{remark}
Note that the class of $\HHess(X,h)$ does not depend on $X$, so long as $X$ is sufficiently generic to guarantee that this scheme has expected codimension.  For this reason, we often write ${[\Hess_\lambda]}$ for this class, where $\lambda$ is the partition corresponding to $h$.

Without such a genericity restriction on $X$, $\HHess(X,h)$ may not have expected codimension.  Indeed, the Hessenberg variety $\mathcal{H}(X,h)$ need not be pure dimensional \cite{t2}, and general Hessenberg varieties have no known closed dimension formula.

In general, the question of which Hessenberg schemes are reduced and pure dimensional of expected codimension depends on both $X$ and $h$.  One situation where this holds is when $X$ is a regular semisimple operator: in this case $\Hess(X,h)$ is smooth and pure dimensional of dimension $\sum_{i=1}^n h(i)-i$.  In fact, this holds for arbitrary Lie type; this is the content of \cite[Theorem 6]{mps}, and the fact that $\HHess(X,h)$ is reduced follows from the proof.  (Their proof works verbatim scheme-theoretically, and shows that a certain subscheme $G_H\subseteq G$ is smooth; the main point is that, in type $A$, the equations defining $G_H$ are the same as those defining the Hessenberg scheme $\HHess(X,h)$, so the latter is also smooth, hence reduced.)  These Hessenberg varieties are connected if and only if $h(i) \geq i+1$ for each $i<n$; see Appendix \ref{sec:connected}.
\end{remark}

\begin{remark} \label{remark:conjugacy}
Let $M'_n$ be the space of trace-zero matrices.  Given a Hessenberg function $h$, let $H\subset M'_n$ be the \define{Hessenberg space} of matrices given by
\begin{eqnarray*}
H = \{M \in M'_n: M_{ij} = 0 \textup{ if } i > h(j)\}.
\end{eqnarray*}
Then the Hessenberg variety of $X$ and $h$ may also be written
\begin{eqnarray*}
\Hess(X,H) = \{[x] \in SL_n/B: x^{-1}Xx \in H\}.
\end{eqnarray*}
>From this description, it is clear that $g\cdot\Hess(X,h) = \Hess(g^{-1}Xg,h)$ for $g\in SL_n$, where $SL_n$ acts on $\Fl=SL_n/B$ by left multiplication.  In particular, the class $[\Hess(X,h)]$ only depends on the conjugacy class of $X$.
\end{remark}

Define the \define{(generalized) Hessenberg coefficients} $m_{w,u}(v)$ to be the integers such that
\begin{eqnarray} 
\Sch_{w}(x;x^v) = \sum_u m_{w,u}(v)\,\Sch_u(x),
\end{eqnarray}
where $x^v = (x_{v(1)},x_{v(2)},\ldots)$ is the permutation of the variables $x$ by $v$.  We are interested in the case $v=w_0\in S_n$ and $w = w(h) = w(\lambda)$, so we will write $\tilde{x} = x^{w_0}$ and $m_{\lambda,u}^{(n)} = m_{w(\lambda),u}(w_0)$.  Thus we have
\begin{eqnarray} \label{Hessenberg coefficients}
{[\Hess_\lambda]} = \sum_{u\in S_n} m_{\lambda,u}^{(n)}\, {[\Omega_u]}.
\end{eqnarray}
Equivalently, since ${[\Omega_{w_0\,u}]}$ is Poincar\'e dual to ${[\Omega_u]}$, we have $m_{\lambda,u}^{(n)} = {[\Hess_\lambda]}\cdot {[\Omega_{w_0\,u}]}$.  Since $\Fl$ is a homogeneous variety, Kleiman's transversality theorem \cite[III.10]{ha} implies that $m_{\lambda,u}^{(n)} \geq 0$.  This is a special case of a simple fact about specializations of Schubert polynomials:
\begin{proposition} \label{prop:positive}
Let $w\in S_n$, and let $w_0$ be the longest permutation in $S_n$.  In the expansion of $\Sch_w(x;x^{w_0})$ in the basis of Schubert polynomials, the coefficient of $\Sch_u(x)$ is nonnegative for all $u \in S_n$.
\end{proposition}

\begin{proof}
Let $D: \Fl(n) \to \Fl(n)$ be the duality automorphism that sends $(F_1 \subset \cdots \subset F_n = \C^n)$ to $( (F_n/F_{n-1})^* \subset \cdots \subset (F_n/F_1)^* \subset F_n^* = \C^n)$.  Then $D^*(x_i) = -x_{n+1-i}$, and $D^{-1}\Omega_w(E_\bullet) = \Omega_{w_0\,w\,w_0}(\tilde{E}_\bullet)$, where $E_\bullet$ is the standard flag in $\C^n$ (with $E_i$ spanned by standard basis vectors $e_1,\ldots,e_i$), and $\tilde{E}_\bullet$ is the opposite flag (with $\tilde{E}_i$ spanned by $e_n,\ldots,e_{n+1-i}$).  Therefore we have
\begin{eqnarray*}
\Sch_{v}(-\tilde{x}) = D^*\Sch_v(x) \equiv \Sch_{w_0\,v\,w_0}(x),
\end{eqnarray*}
modulo the ideal defining $H^*\Fl(n)$.  Now using the identity
\begin{eqnarray} \label{eqn:switch-xy}
\Sch_w(x;y) = \sum_{v^{-1}u=w} \Sch_u(x)\cdot\Sch_v(-y)
\end{eqnarray}
from \cite[(6.3)]{mac}, we see that
\begin{eqnarray} \label{eqn:exp}
\Sch_w(x;\tilde{x}) \equiv \sum_{v^{-1}u=w} \Sch_u(x)\cdot\Sch_{w_0\,v\,w_0}(x).
\end{eqnarray}
Since the ideal is just the submodule of $\Z{[x_1,\ldots,x_n]}$ spanned by $\{\Sch_u\,|\, u\not\in S_n\}$, and the right hand side of (\ref{eqn:exp}) has a positive expansion in the basis of Schubert polynomials, the proposition follows.
\end{proof}

Proposition \ref{prop:positive} does not hold for specializations at other permutations.  For example, $\Sch_{213}(x;x^{231}) = x_1 - x_2 = 2\,\Sch_{213}(x) - \Sch_{132}(x)$.

Hessenberg coefficients exhibit a symmetry corresponding to the symmetry sending a partition $\lambda$ to the partition $\lambda'$ with rows and columns transposed.
\begin{proposition} \label{prop:symmetry}
For $w,u\in S_n$, we have $m_{w,u}(w_0) = m_{w^{-1},\,w_0\,u\,w_0}(w_0)$.  In particular, since $w(\lambda') = w(\lambda)^{-1}$, we have $m_{\lambda,u}^{(n)} = m_{\lambda',\,w_0\,u\,w_0}^{(n)}$.
\end{proposition}

\begin{proof}
Again, let $D: \Fl(n) \to \Fl(n)$ be the duality automorphism.  We have
\begin{eqnarray*}
D^*\Sch_w(x;\tilde{x}) = \Sch_w(-\tilde{x};-x) &\equiv& \sum_{u\in S_n} m_{w,u}(w_0)\, D^*\Sch_u(x) \\
&\equiv& \sum_{u\in S_n} m_{w,u}(w_0)\, \Sch_{w_0\,u\,w_0}(x).
\end{eqnarray*}
The identity $\Sch_w(x;y) = \Sch_{w^{-1}}(-y;-x)$ from \cite[(6.4)]{mac} shows $\Sch_w(-\tilde{x};-x) = \Sch_{w^{-1}}(x;\tilde{x})$, so the proposition follows.
\end{proof}

\section{Dominant permutations} \label{sec:perms}

The permutations indexing Hessenberg schemes form a very special class called \define{dominant permutations}.  Given a partition $\lambda$, one can construct a permutation $w=w(\lambda)$ in $S_n$, for $n\geq \lambda_i+i$ for all $i$, as follows:\footnote{In the literature, the notation $w(\lambda)$ is sometimes used for the \emph{Grassmannian permutation} corresponding to $\lambda$, whereas for us this denotes a \emph{dominant permutation}.  With the exception of $\lambda = (1,\ldots,1,0,\ldots,0)$, these two notions are different.} Write $w = {[w(1), w(2), \ldots, w(n)]}$, setting $w(1)=\lambda_1+1$, and taking $w(i)$ to be the $(\lambda_i+1)$st element of $\{1,\ldots,n\}\setminus \{w(1),\ldots,w(i-1)\}$.  The permutations constructed this way are called dominant.  (Dominant permutations can also be characterized as those permutations whose \emph{diagrams} are Young diagrams; see \cite[\S1]{mac}, \cite{flags} for the notion of the diagram of a permutation.)

The Schubert polynomials corresponding to dominant permutations have a simple ``monomial'' form.  Specifically,
\begin{eqnarray*}
\Sch_{w(\lambda)}(x;y) = \prod_{(i,j)\in\lambda}(x_i - y_j) .
\end{eqnarray*}
This can be proved combinatorially or algebraically (\cite[(6.14)]{mac}), but we give a geometric reason.

\begin{proposition} \label{prop:geometric-monomial}
Let $\lambda$ be a partition such that $\lambda_i \leq n-i$, and let $S^{(2)}_{\bullet} = \pi_2^*S_{\bullet}$ and $Q^{(1)}_{\bullet} = \pi_1^*Q$ be the universal bundles on $\Fl(n)\times \Fl(n)$.  Then $\OOmega_{w(\lambda)} \subseteq \Fl \times \Fl$ is the zero locus of a section of the vector bundle
\begin{eqnarray*}
K = \ker\left(\bigoplus_{p=1}^n \Hom(S^{(2)}_p, Q^{(1)}_{\lambda_p}) \xrightarrow{f} \bigoplus_{p=1}^n \Hom(S^{(2)}_p, Q^{(1)}_{\lambda_{p+1}}) \right),
\end{eqnarray*}
where the map $f$ is
\begin{eqnarray*}
(S_p \xrightarrow{\phi} Q_{\lambda_p}) \mapsto (S_p \xrightarrow{\phi} Q_{\lambda_p} \to Q_{\lambda_{p+1}}) - (S_{p-1} \subset S_p \xrightarrow{\phi} Q_{\lambda_p}).
\end{eqnarray*}
\end{proposition}

The proof is similar to that of \cite[Proposition 7.5]{flags}.  Since the codimension of $\OOmega_{w(\lambda)}$ is equal to the rank of $K$, it follows that ${[\OOmega_{w(\lambda)}]} = c_{top}(K)$, which is 
\begin{eqnarray*}
c_{top}(K) &=& \frac{c_{top}\left(\bigoplus \Hom(S_p,Q_{\lambda_p})\right)}{c_{top}\left(\bigoplus \Hom(S_p,Q_{\lambda_{p+1}})\right)} \\
&=& \prod_{(i,j)\in\lambda}(x_i - y_j).
\end{eqnarray*}
(The dominant Schubert varieties $\Omega_{w(\lambda)}$ are actually smooth, so the restriction of $K$ is the normal bundle.  There are many ways of showing that $\Omega_{w(\lambda)}$ is smooth; see, e.g., \cite[p.\ 37]{dmr} for an explicit statement of this fact.)

In fact, the converse is true: only Schubert polynomials with this monomial form correspond to dominant permutations.  To see this, we use the combinatorics of \define{pipe dreams}, which give a combinatorial model for the terms in a Schubert polynomial.  See \cite{ms} or \cite{km} for details and examples.

A pipe dream in $S_n$ is a subset of boxes from the $n \times n$ grid, which we represent as boxes filled with $+$.  Each pipe dream defines a permutation as follows.  Fill each blank box in the pipe dream with the (oriented) tile 
\begin{picture}(10,10)(0,0)
\multiput(0,10)(0,-10){2}{\line(1,0){10}}
\multiput(0,10)(10,0){2}{\line(0,-1){10}}
\multiput(1,6)(1,1){4}{\circle*{1}}
\multiput(6,1)(1,1){4}{\circle*{1}}
\end{picture}.  If the strand leaving the $i^{th}$ row of the left side of the pipe dream (reading down) enters the $j^{th}$ column on top of the pipe dream (reading right), then the permutation $w$ corresponding to the pipe dream has $w(i)=j$.  A pipe dream is said to be reduced if no two of its strands cross more than once.

The key theorem we use is due to Billey--Jockusch--Stanley \cite{bjs} and Fomin--Stanley \cite{fs} in the single case, and Fomin--Kirillov \cite{fk} in the double case; see also \cite[Corollary 2.1.5]{km}.
\begin{theorem} {\cite{bjs,fs,fk}} \label{thm:pipedream}
The (single or double) Schubert polynomial $\mathfrak{S}_w$ 
is the sum of monomials over all reduced pipe dreams for $w$.
\end{theorem}

For our purposes, we need no further detail of that theorem: we only use the fact that a permutation with more than one reduced pipe dream does not have a monomial Schubert polynomial.

\begin{proposition} \label{prop:not monomial}
The permutation $w$ is dominant if and only if $\mathfrak{S}_w$ is a monomial.
\end{proposition}

\begin{proof}
Proposition \ref{prop:geometric-monomial} showed that each dominant permutation has a monomial Schubert polynomial. Theorem \ref{thm:pipedream} states that a Schubert polynomial is monomial if and only if the corresponding permutation $w$ has exactly one reduced pipe dream.  A counting argument will show that there are no non-dominant $w$ with exactly one reduced pipe dream.

\cite[Proposition 16.24, Exercise 16.7]{ms} say that each permutation has a unique reduced pipe dream whose $+$'s are top-justified and a unique reduced pipe dream whose $+$'s are left-justified, respectively.  If $w$ is a permutation with a unique reduced pipe dream then its pipe dream must be both top-justified and left-justified.  This means the $+$'s of the pipe dream form a Young diagram.  Hence the Young diagrams count both permutations with unique reduced pipe dreams and dominant permutations.  Since dominant permutations are a subset of permutations with unique reduced pipe dreams, we conclude that dominant permutations are exactly those with a unique reduced pipe dream, equivalently with monomial Schubert polynomials.
\end{proof}

\section{Computing classes of Hessenberg varieties} \label{sec:computing}

In this section and the next, we prove formulas for the Hessenberg coefficients $m_{\lambda,u}^{(n)}$ in several cases.  The main tool we use is the following basic formula for expressing an arbitrary polynomial in terms of Schubert polynomials.  (See \cite[(4.14)]{mac}.)
\begin{lemma} \label{lemma:divdiffcoeffs}
For any homogeneous polynomial $P=P(x)\in\Z{[x_1,\ldots,x_n]}$ of degree $d$, we have
\begin{eqnarray}
P = \sum_{\ell(v)=d} (\partial_v P)\,\Sch_v(x),
\end{eqnarray}
where the sum is over length $d$ permutations in $S_\infty$.
\end{lemma}
We are only interested in the coefficients for those $v$ which are in $S_n$, since $\Sch_v(x) \equiv 0$ in $H^*(\Fl(n))$ if $v \not\in S_n$.

The first case we consider is that of \emph{stable range} Hessenberg classes.  We say a Hessenberg variety in $\Fl(n)$ is in the stable range if the corresponding permutation $w(h)$ is in $S_k \subset S_n$, with $2k\leq n$.  Equivalently, the corresponding partition $\lambda$ is contained in the staircase $(k-1,k-2,\ldots,1,0)$.  We will see that $m_{\lambda,u}^{(n)}$ is always $1$ or $0$ in the stable range.

To compute these coefficients, we use Lemma \ref{lemma:divdiffcoeffs} together with the fact that double Schubert polynomials can be computed with respect to the $y$ variables as well as the $x$ variables.  When divided difference operators are computed with respect to $y$, we obtain the formula
\begin{eqnarray*}
\partial_i^y \Sch_w(x;y) = \left\{ \begin{array}{cl} \Sch_{s_iw}(x;y) & \text{if } \ell(w) >\ell(s_iw) ; \\ 0 &\text{if } \ell(w) < \ell(s_iw). \end{array} \right.
\end{eqnarray*}

\begin{lemma} \label{lemma:x-y-divdiff}
Suppose a polynomial $P(x;y)$ is symmetric in the variables $x_{n-i}$ and $x_{n-i+1}$.  Then
\begin{eqnarray}
-\partial^y_i P(x;y)|_{y=\tilde{x}} = \partial^x_{n-i}P(x;\tilde{x}).
\end{eqnarray}
Similarly, if $P$ is symmetric in $y_{n-i}$ and $y_{n-i+1}$, then
\begin{eqnarray}
\partial^x_i P(x;y)|_{y=\tilde{x}} = \partial^x_i P(x;\tilde{x}).
\end{eqnarray}
\end{lemma}

\begin{proof}
Indeed,
\begin{eqnarray*}
\partial^y_i P(x;y)|_{y=\tilde{x}} &=& \left.\frac{P(x;y) - P(x;s_i(y))}{y_i - y_{i+1}}\right|_{y=\tilde{x}} \\
&=& \left.\frac{P(x;y) - P(s_{n-i}(x);s_i(y))}{y_i - y_{i+1}}\right|_{y=\tilde{x}} \\
&=&\frac{P(x;\ldots,x_{n-i+1},x_{n-i},\ldots) - P(s_{n-i}(x);\ldots,x_{n-i},x_{n-i+1},\ldots)}{x_{n-i+1} - x_{n-i}} \\
&=& -\partial^x_{n-i}P(x;\tilde{x}).
\end{eqnarray*}
A similar calculation proves the second statement.
\end{proof}

\begin{remark}
In fact, the same type of calculation proves the following: Let $w$ be a permutation, and suppose $P(x;y)$ is symmetric in $x_{w(i)}$ and $x_{w(i+1)}$.  
\begin{itemize}
\item If $w(i+1) = w(i)+1$, then $\partial^y_i P(x;y)|_{y = x^w} = \partial^x_{w(i)} P(x;x^w)$.
\item If $w(i+1) = w(i)-1$, then $\partial^y_i P(x;y)|_{y = x^w} = -\partial^x_{w(i+1)} P(x;x^w)$.
\end{itemize}
\end{remark}

If $k$ is the largest index such that $x_k$ appears in $P(x;y)$, then the lemma applies to $\partial^y_i$ for all $i<n-k$; if $l$ is the largest index such that $y_l$ appears in $P(x;y)$, then the lemma applies to $\partial^x_i$ for all $i<n-l$.  If $k+l<n$, then, the lemma gives an expression for $\partial^x_i P(x;\tilde{x})$ for all $i$.  Therefore we frequently assume a polynomial $P$ satisfies the following:
\renewcommand{\theenumi}{\fnsymbol{enumi}}
\begin{enumerate}
\item If $k$ is the largest index such that $x_k$ appears in $P(x_1,\ldots,x_n;y_1,\ldots,y_n)$, and $l$ is the largest index such that $y_l$ appears, then $k+l<n$. \label{independence}
\end{enumerate}
\renewcommand{\theenumi}{\roman{enumi}}

Lemma \ref{lemma:x-y-divdiff} together with Condition (\ref{independence}) immediately imply :
\begin{lemma} \label{lemma:divdiffcomp}
Suppose $\Sch_w(x_1,\ldots,x_n;y_1,\ldots,y_n)$ satisfies (\ref{independence}) and let $k$ be the largest index such that $x_k$ appears.  Then
\begin{eqnarray}
\partial^x_i\Sch_w(x;\tilde{x}) = \left\{\begin{array}{cl} 
\Sch_{w\,s_i}(x;\tilde{x}) &\text{if }i\leq k \text{ and }w\,s_i<w ; \\
\Sch_{s_{n-i}\,w}(x;\tilde{x}) &\text{if }i>k \text{ and }s_{n-i}\,w<w; \\
0 &\text{otherwise}. \end{array}\right.
\end{eqnarray}
\end{lemma}

When $\Sch_{w\,s_i}(x;y)$ or $\Sch_{s_{n-i}\,w}(x;y)$ also satisfy (\ref{independence}), we can apply divided differences again using Lemma \ref{lemma:divdiffcomp}.  In the case where the Schubert polynomials $\Sch_u$ satisfy (\ref{independence}) for all permutations $u$ with $u\leq w$, this method calculates all coefficients $\partial^x_{w'}\Sch_w(x;\tilde{x})$, for $\ell(w')=\ell(w)$.

\begin{theorem} \label{thm:stableclasses}
Let $w=w(\lambda)$.  If $w\in S_k \subset S_n$ with $2k\leq n$, then
\begin{eqnarray}
{[\Hess_{\lambda}]} = \sum_{v^{-1}u=w} {[\Omega_{u\,w_0\,v\,w_0}]}
\end{eqnarray}
in $H^*(\Fl(n))$, where the sum is over permutations $u,v$ such that $v^{-1}u=w$ and $\ell(u)+\ell(v)=\ell(w)$.
\end{theorem}

\begin{proof}
Let $w$ be contained in $S_k \subset S_n$, with $2k\leq n$.  Then $\Sch_w(x;y)$ satisfies (\ref{independence}), as do all $\Sch_u$ for $u\leq w$.  For any $u\leq w$, write $v^{-1}=w\,u^{-1}$, so $w=v^{-1}u$ and $\ell(u)+\ell(v)=\ell(w)$.  Both $u$ and $v$ lie in $S_k \subset S_n$ and so $w_0\,v\,w_0$ commutes with $u$.  Hence the corresponding divided difference operators also commute.  For each such $u$ and $v$, we have
\begin{eqnarray*}
\partial_{u\,w_0\,v\,w_0}\Sch_w(x;\tilde{x}) &=& \partial_u\partial_{w_0\,v\,w_0}\Sch_w(x;\tilde{x}) \\
&=& (-1)^{\ell(v)}\partial^x_u\partial^y_v\Sch_w(x;y)|_{y=\tilde{x}} \\
&=& 1,
\end{eqnarray*}
and $\partial_{w'}\Sch_w(x;\tilde{x}) = 0$ for all other $w'$, using Lemma \ref{lemma:divdiffcomp} repeatedly.  Lemma \ref{lemma:divdiffcoeffs} now proves the claim.
\end{proof}

Theorem \ref{thm:stableclasses} can also be proved geometrically, as follows.  Using Formula (\ref{eqn:exp}), we see that for any $w\in S_n$,
\begin{eqnarray}
\Sch_w(x;\tilde{x}) = \sum_{v^{-1}u=w} {[\Omega_u]}\cdot{[\Omega_{w_0\,v\,w_0}]}.
\end{eqnarray}
If $w$ is in the stable range, then ${[\Omega_u]}\cdot{[\Omega_{w_0\,v\,w_0}]} = {[\Omega_{u\,w_0\,v\,w_0}]}$.  Indeed, this follows from a more general fact:
\begin{proposition}
Let $a,b \in S_n$ be permutations such that $a$ fixes $k+1,\ldots,n$ and $b$ fixes $1,\ldots,k$, for some $k\leq n$.  Then ${[\Omega_a]}\cdot{[\Omega_b]} = {[\Omega_{a\,b}]}$.
\end{proposition}
\begin{proof}
The Schubert cells $\Omega^o_a(b\cdot E_{\bullet})$ and $\Omega^o_b(a\cdot E_{\bullet})$ intersect transversally in the cell $\Omega^o_{a\,b}(E_{\bullet})$, and this suffices to compute the product of Schubert classes: see \cite[Appendix B]{yt}.
\end{proof}

\section{A combinatorial formula} \label{sec:formula}

Here we give a formula for $m_{\lambda,u}^{(n)}$ in the case $\lambda = (k)$ consists of a single row.  The corresponding Hessenberg functions have $h(1) = n-k$, and $h(i) = n$ for $i>1$; the corresponding permutation $w(\lambda)$ is a cycle.  Note that the specialized double Schubert polynomial is $\Sch_{w(\lambda)}(x;\tilde{x}) = (x_1 - x_n)(x_1 - x_{n-1})\cdots(x_1 - x_{n+1-k})$ in this case.  (The case where $\lambda$ is a single column can be obtained via the symmetry of Proposition \ref{prop:symmetry}.)

We state our combinatorial formula in terms of paths in the Bruhat graph.  We use $s_{ij}$ to denote the transposition exchanging $i$ and $j$.

\begin{definition}
A \define{restricted path} (with $p$ terms) from $u$ to $v$ is a sequence
\begin{eqnarray*}
u = u_1, u_2, u_3, \ldots, u_p = v,
\end{eqnarray*}
where
\begin{enumerate}
\item $\ell(u_{i+1}) = \ell(u_i)-1$, and \label{path-condition1}
\item $u_{i+1} = u_{i}\cdot s_{1j}$. \label{path-condition2}
\end{enumerate}
\end{definition}

\begin{remark}
Because of the special form of the transpositions in condition \eqref{path-condition2}, not every path in (strong or weak) Bruhat order is a restricted path.  In fact, restricted paths are precisely paths in ``$1$-Bruhat order,'' in the terminology of \cite{Soflagspieri}.
\end{remark}

\begin{theorem} \label{thm:rows}
Let $w\in S_n$ have length $k$.  The Hessenberg coefficient $m_{(k),w}^{(n)}$ is the number of restricted paths from $w$ to a permutation of the form $s_{n-p} s_{n-p-1}\cdots s_{n-k}$, where $p$ is the number of terms in the path.
\end{theorem}

For example, if $w = s_2 s_1 = {[3,1,2,4]}$, the two paths giving $m_{(2),w}^{(4)} = 2$ are
\begin{eqnarray*}
\underline{3}\;1\;\underline{2}\;4 \to \underline{2}\;\underline{1}\;3\;4 \to 1\;2\;3\;4, \quad \text{i.e., } s_2 s_1 \to s_1 \to e,
\end{eqnarray*}
and
\begin{eqnarray*}
\underline{3}\;\underline{1}\;2\;4 \to 1\;3\;2\;4, \quad \text{i.e., } s_2 s_1 \to s_2.
\end{eqnarray*}
(The underlined numbers are the ones to be swapped.)  Note that the $1$-term path consisting of $w$ itself is not of the form prescribed in the theorem.

\medskip
The target of the paths described in the theorem is the cycle
\begin{eqnarray*}
\pi^{(p)} = {[1,2,\ldots,n-k-1,n-p+1,n-k,\ldots,n-p,n-p+2,\ldots,n]}.
\end{eqnarray*}
In fact, restricted paths with this target must satisfy an additional constraint, which helps with computation:
\begin{proposition}
If $\ell(w)=k$, and $w = u_1, \ldots, u_p = \pi^{(p)}$ is a restricted path, then at the $i$th step we have 
\renewcommand{\theenumi}{ii$'$}
\begin{enumerate}
\item $u_{i+1} = u_i\cdot s_{1j}$ with $2\leq j\leq n+1-i$. \label{n-res}
\end{enumerate}
\end{proposition}
\renewcommand{\theenumi}{\roman{enumi}}

\begin{proof}
In fact, for such a path, $u_i \in S_{n+1-i} \subset S_n$.  Clearly this is true for $u_p = \pi^{(p)}$.  Working backwards, suppose $u_{i+1} \in S_{n-i}$ but $u_i \not\in S_{n+1-i}$.  Since $u_i = u_{i+1}\cdot s_{1j}$, this happens if and only if $j>n+1-i$.  But then $u_{i+1}$ fixes $n+1-i$ and $j$.   Since $u_i$ has the inversions $j>u_i(1)$ and $j>n+1-i$, we conclude $\ell(u_i) \geq \ell(u_{i+1}) + 2$, a contradiction.
\end{proof}

\noindent
Call a restricted path satisfying (\ref{n-res}) an \define{$n$-restricted path}.

\medskip
Before proving the theorem, we need some lemmata.  The first is a standard fact from the theory of normal forms in Coxeter groups (cf.\ \cite[Corollary 2.4.6]{bb}).

\begin{lemma} \label{lemma:factorization}
Any permutation $w\in S_n$ can be factored uniquely
as a product of the form 
\begin{eqnarray*}
w = \prod_{i=2}^{n}(s_{i-1}s_{i-2}s_{i-3} \cdots s_{i-j_i})
\end{eqnarray*}
for some $0\leq j_i < i$, where the $i^{th}$ factor is $e$ if $j_i = 0$.
\end{lemma}

We also use the following from \cite[Proposition 2.13, page 27]{mac}.

\begin{proposition} \label{Mac formula}
If $f = \sum \alpha_i x_i$ for coefficients $\alpha_i \in \mathbb{C}$, then
\begin{eqnarray*}
\partial_w(fg) = w(f) \partial_wg + \sum(\alpha_i - \alpha_j) \partial_{ws_{ij}}g,
\end{eqnarray*}
summed over all pairs $i<j$ such that $\ell(ws_{ij}) = \ell(w)-1$, where
$s_{ij}$ is the transposition
that interchanges $i$ and $j$.
\end{proposition}

\begin{proposition} \label{one decomposition}
If $g = \prod_{i=n+1-k}^{n-1}(x_1-x_i)$, $w$ has length $k$, and $w \not > s_{n-1}$ then
\[\partial_w((x_1-x_n)g) =\sum \partial_{ws_{1j}}g\]
where the sum is over all transpositions $s_{1j}$
 such that $\ell(ws_{1j})=\ell(w)-1$.
\end{proposition}

\begin{proof}
Proposition \ref{Mac formula} reduces in our case to
\[\partial_w\left( (x_1-x_n) g \right) = w(x_1-x_n) \partial_wg + \sum \partial_{ws_{1j}}g
+ \sum \partial_{ws_{in}}g,\]
where the sums are taken over transpositions $s_{ij}$ with $\ell(ws_{ij}) = \ell(w)-1$.
Since $\ell(w) = k$, we have $\partial_w g = 0$ by degree considerations.  Note that $s_{in}>s_{n-1}$ since $s_{in}=s_is_{i+1}\cdots s_{n-2}s_{n-1}s_{n-2} \cdots s_{i+1}s_i$.  This means $ws_{in}>s_{n-1}$ and so $w \not > ws_{in}$ in the Bruhat order.  Multiplication by a transposition satisfies either $ws_{in}>w$ or $w>ws_{in}$ \cite[\S1]{mac}.  We conclude that $\ell(ws_{in}) > \ell(w)$ for all $s_{in}$, and the claim follows.  
\end{proof}

\begin{lemma} \label{no $s_n$}
Let $w$ be a permutation with length $k$.
If $w \geq s_{n-1}$ then one of the following holds:
\begin{itemize}
\item $k<n-1$, $w = s_{n-1} s_{n-2} \cdots s_{n-k}$, and $\partial_w \left(\prod_{i=n+1-k}^{n} (x_1-x_i)\right) 
=1$;
\item $k=n-1$, $w =  s_{n-1} s_{n-2} \cdots s_{1}$, and $\partial_w \left(\prod_{i=2}^{n} (x_1-x_i)\right) 
=n$; or
\item $\partial_w \left(\prod_{i=n+1-k}^{n} (x_1-x_i)\right) =0$.
\end{itemize}
\end{lemma}

\begin{proof}
By Lemma \ref{lemma:factorization}, factor $w$ uniquely
in the form $\prod_{i=2}^{n}(s_{i-1}s_{i-2}s_{i-3} \cdots s_{i-j_i})$.  Note that
$\partial_{n-j}\left(\prod_{i=n+1-k}^{n} (x_1-x_i)\right) = 0$ unless $j=n-1$ or $j=k$.
Observe that 
\begin{equation}\label{one div diff}
\partial_{n-k}\left(\prod_{i=n+1-k}^{n} (x_1-x_i)\right) = \left(\prod_{i=n+2-k}^{n} (x_1-x_i)\right).
\end{equation}
We may write
\[\partial_w \left(\prod_{i=n+1-k}^{n} (x_1-x_i)\right) = \partial_{w'} \partial_{n-1} \cdots
\partial_{n-j} \left(\prod_{i=n+1-k}^{n} (x_1-x_i)\right)\]
for a unique permutation $w'$ such that $w' \not \geq s_{n-1}$.  
If $j=k < n-1$, then $w'=e$, and by induction on Equation \eqref{one div diff}, the divided difference equals $1$.  If $j=n-1 \neq k$, then the length of $w$ is greater than the degree of the polynomial $\left(\prod_{i=n+1-k}^{n} (x_1-x_i)\right)$ and so the divided difference is $0$.  If $j=n-1=k$ then Lemma \ref{lemma:$n-1$} below applies. 
\end{proof}

\begin{lemma} \label{lemma:$n-1$}
If $w = s_{n-1}s_{n-2} \cdots s_1$ then $\partial_w \left( \prod_{i=2}^n (x_1-x_i) \right) = n$.
\end{lemma}

\begin{proof}
Define $w_j$ to be the product $s_j s_{j-1} \cdots s_1$.  By induction on $j$, we will prove the following claim (valid for $0\leq j\leq n-1$):
\begin{eqnarray} \label{eqn:induct-n}
\partial_{w_{j}} \left( \prod_{i=2}^n (x_1-x_i) \right) = \sum_{i=1}^{j+1} \prod_{m=j+2}^n (x_i-x_m).
\end{eqnarray}
The base case $j=0$ is trivial.  For the induction step, assume $0<j<n$ and compute:
\begin{eqnarray*}
\partial_{w_j} \left( \prod_{i=2}^n (x_1 - x_i) \right) &=& \partial_j \circ \partial_{w_{j-1}} \left( \prod_{i=2}^n (x_1 - x_i) \right) \\
&=& \partial_j \left( \sum_{i=1}^{j} \prod_{m=j+1}^n (x_i-x_m) \right) \\
&=& \sum_{i=1}^{j-1} \partial_j \left( \prod_{m=j+1}^n (x_i - x_m) \right) + \partial_j \left(\prod_{m=j+1}^n (x_j - x_m) \right) \\
&=& \sum_{i=1}^{j-1} \prod_{m=j+2}^n (x_i - x_m) + \prod_{m=j+2}^n (x_j - x_m) + \prod_{m=j+2}^n (x_{j+1} - x_m),
\end{eqnarray*}
using the facts that for $i<j<n$
\begin{eqnarray*}
\partial_j \prod_{m=j+1}^n (x_i-x_m) = \prod_{m=j+2}^n (x_i-x_m)
\end{eqnarray*}
and
\begin{eqnarray*}
\partial_j \prod_{m=j+1}^n (x_j-x_m) = \frac{(x_j-x_{j+1})\prod_{m=j+2}^n (x_j-x_m) - (x_{j+1}-x_j) \prod_{m=j+2}^n (x_{j+1}-x_m)}{x_j - x_{j+1}}.
\end{eqnarray*}

Note that $w=w_{n-1}$, and the case $j=n-1$ of (\ref{eqn:induct-n}) gives
\begin{eqnarray*}
\partial_w  \prod_{i=1}^n (x_1-x_i)  = \sum_{i=1}^{n} 1 = n.
\end{eqnarray*}
\end{proof}

We are now ready to prove the main theorem.

\begin{proof}[Proof of Theorem \ref{thm:rows}]
We break the proof into two cases: $w \not\geq s_{n-1}$ and $w \geq s_{n-1}$.  Let $RP(n,w,p)$ be the set of $p$-term $n$-restricted paths from $w$ to $s_{n-p} \cdots s_{n-k}$, and let $RP(n,w) = \bigcup_{p=1}^{k+1} RP(n,w,p)$.  (The $p=k+1$ case corresponds to $k+1$-term paths from $w$ to $e$.)  Thus the claim is that $m_{(k),w}^{(n)} = \# RP(n,w)$.

\medskip

{\it Case 1: $w \not\geq s_{n-1}$.}  This means $w\in S_{n-1}\subset S_n$.  Proposition \ref{one decomposition} says 
\begin{eqnarray*}
\partial_w \left(\prod_{i=n+1-k}^n (x_1-x_i)\right) = \sum \partial_{ws_{1j}} \left(\prod_{i=n+1-k}^{n-1}(x_1-x_i)\right).
\end{eqnarray*} 
It follows that for $w\not\geq s_{n-1}$ we have
\begin{eqnarray} \label{eqn:recur}
m_{(k),w}^{(n)} = \sum m_{(k-1),w'}^{(n-1)},
\end{eqnarray}
where the sum is over all $w' = w s_{1j}$ with $2\leq j\leq n$ and $\ell(w') = \ell(w)-1 = k-1$.  (Thus $(w,w')$ is a restricted path.)  The numbers $\# RP(n,w)$ also satisfy the recursion in (\ref{eqn:recur}), since
\begin{eqnarray*}
RP(n,w,p) \quad \text{  and  }\quad \bigcup_{w'} RP(n-1,w',p-1)
\end{eqnarray*}
are in bijection.  To see this, note that $w=u_1, u_2, \ldots, u_p = s_{n-p}\cdots s_{n-k}$ is $n$-restricted if and only if $u_2, \ldots, u_p$ is $(n-1)$-restricted, with $u_2 = w' = w s_{1j}$ as in the sum.

When $w'\not\geq s_{n-2}$, we may replace $(w,n,k,p)$ with $(w',n-1,k-1,p-1)$ and iterate the recursion by repeating Case 1.  To complete the proof, we must check that $m_{(k),w}^{(n)} = \# RP(n,w)$ at the base of the recursion, which is where the assumption $w\not\geq s_{n-1}$ fails.

\medskip
{\it Case 2: $w\geq s_{n-1}$.}  In this case, we know $m_{(k),w}^{(n)}$ by Lemma \ref{no $s_n$}, and we must check that these numbers agree with the numbers of $n$-restricted paths.  Note that in one-line notation, the permutation $s_{n-1} \cdots s_{n-k}$ is the cycle ${[1,2,\ldots, n-k-1, n, n-k,\ldots, n-1 ]}$.

If $k<n-1$ and $w = s_{n-1} \cdots s_{n-k}$, then there is a $1$-term $n$-restricted path from $w$ to itself.  There are no other $n$-restricted paths from $w$ to permutations of smaller length, since transposing the first entry of $w = {[1,2,\ldots, n-k-1, n, n-k,\ldots, n-1 ]}$ always increases length.  Therefore $\# RP(n,w) = 1$.  

If $k=n-1$ and $w = s_{n-1}\cdots s_1 = {[n,1,2,\ldots,n-1]}$ then the first step in an $n$-restricted path to a permutation $s_{n-p} \cdots s_1$ must be to $w s_{1n} = s_{n-2}\cdots s_1$; this can be seen by writing $s_{1j} = s_1 s_2 \cdots s_{j-1} \cdots s_2 s_1$.  Thus for each $p = 1, \ldots, n$, there is exactly one $n$-restricted path with $p$ terms from $w$ to $s_{n-p}\cdots s_1$, namely
\begin{eqnarray*}
w = s_{n-1}\cdots s_1,\; s_{n-2} \cdots s_1,\; \ldots,\; s_{n-p} \cdots s_1 .
\end{eqnarray*}
So $\# RP(n,w) = n$.

Finally, suppose $w$ is not a cycle, and write $w = w's_{n-1} s_{n-2} \cdots s_j$, for some $w' \neq e$, as in Lemma \ref{lemma:factorization}.  Since $\ell(w) = k \leq n-1$, we must have $j>1$.  We have already seen that right multiplication by $s_{1m}$ cannot reduce the length of $s_{n-1}\cdots s_j$, so the only $n$-restricted path from $w$ is the trivial path; however $w$ is not of the desired form, so $RP(n,w) = \emptyset$.
\end{proof}

\section{Regular nilpotent Hessenberg varieties} \label{limits}

A linear operator $X$ is \define{regular nilpotent} if it is nilpotent and its Jordan form consists of a single block.  (Thus $X^n = 0$, but $X^{n-1} \neq 0$.)  A Hessenberg scheme $\HHess(X,h)$ is said to be \define{regular semisimple} or \define{regular nilpotent} if $X$ has the corresponding property.  It is known that both of these have dimension $d=\sum_{i=1}^n h(i)-i$ from \cite[Theorem 8]{mps} and \cite[Theorem 10.2]{ST}, and that regular nilpotent varieties have a unique component of this dimension \cite[Theorem 10.2]{ST}.  The following lemma then implies that regular nilpotent varieties are always irreducible.

\begin{lemma} \label{lemma:nilp-dim}
A regular nilpotent Hessenberg variety is pure-dimensional of dimension $d=\sum_{i=1}^n h(i)-i$.
\end{lemma}
\begin{proof}
Using the setup of Remark \ref{remark:conjugacy}, there is a family of Hessenberg schemes $\HHess(h) \subset \Fl(n) \times M'_n$, whose fiber over $X\in M'_n$ is $\HHess(X,h)$.  The generic fiber is a regular semisimple Hessenberg variety, so it has pure dimension $d$.  The lemma is a consequence of the theorem on the dimension of fibers \cite[Ex.\ II.3.22]{ha}, applied to the projection $\HHess \to M'_n$.
\end{proof}

The degeneracy locus formula therefore applies to regular nilpotent Hessenberg schemes:
\begin{corollary} \label{cor:nilpotent class}
If $X_S$ is regular semisimple and $X_N$ is regular nilpotent, then for each Hessenberg function $h$, we have 
\begin{eqnarray*}
[\Hess(X_S,h)] = [\HHess(X_S,h)] = [\HHess(X_N,h)] = \Sch_{w(\lambda)}(x;\tilde{x})
\end{eqnarray*}
in $H^*(\Fl)$.  In particular, if $m^{(n)}_{\lambda,u}$ are the Hessenberg coefficients defined by Equation \eqref{Hessenberg coefficients}, then
\begin{eqnarray*}
{[\HHess(X_N,h)]} = \sum_{u \in S_n} m_{h,u}^{(n)}{[\Omega_u]}.
\end{eqnarray*}
\end{corollary}

\begin{remark}
One can see $[\HHess(X_S,h)] = [\HHess(X_N,h)]$ directly, without knowing the formula.  In fact, it is not hard to show that any regular nilpotent Hessenberg scheme is a flat limit of regular semisimple ones inside $\Fl(n)\times M'_n$; it follows that their classes are equal.
\end{remark}

When $h$ corresponds to the partition $(k)$, then the results of the previous section identify which regular nilpotent Hessenberg schemes are reduced.  The following lemma is a step to identifying the non-reduced schemes.

\begin{lemma} \label{lemma:nilp n-1}
Let $g$ be the Hessenberg function defined by $g(1)=1$ and $g(i)=n$ for all $i \neq 1$, and let $h$ be the Hessenberg function with $h(i)=n-1$ for all $i \neq n$ and $h(n)=n$.  (The partition corresponding to $g$ is a column of length $n-1$, and the one corresponding to $h$ is a row of length $n-1$.)  Let $X$ be a regular nilpotent matrix, let $\{e_1,\ldots,e_n\}$ be a basis for $V$ putting $X$ in Jordan normal form, and let $E_\bullet$ be the corresponding flag.  Then
\begin{eqnarray*}
\Hess(X,g) = \Omega_u(E_\bullet) &\text{for}& u = s_1 s_2 \cdots s_{n-1}, \quad \text{and} \\
\Hess(X,h) = \Omega_v(E_\bullet) &\text{for}& v = s_{n-1}\cdots s_2 s_1.
\end{eqnarray*}
(Each of these Schubert varieties is an embedding of $\Fl(n-1)$ in $\Fl(n)$.)
\end{lemma}

\begin{proof}
In the chosen basis, $X$ is the regular nilpotent matrix which is $0$ except for the entries in position $(i,i+1)$, where it is $1$.  (Note that the dependence of $\Hess(X,h)$ on choice of basis is the same as that of $\Omega_v(E_\bullet)$; see Remark \ref{remark:conjugacy}.)

The Hessenberg variety $\Hess(X,g)$ is defined by the condition that $X V_1 \subseteq V_1$, which is equivalent to the condition that $V_1$ be spanned by an eigenvector for $X$.  The only eigenvalue for $X$ is zero, and the corresponding eigenspace is $\ker X = E_1$.  Hence $\Hess(X,g)$ is defined by $V_1 = E_1$, and this is precisely the condition defining $\Omega_u$.

Similarly, the Hessenberg variety $\Hess(X,h)$ is characterized by the condition that $XV_{n-1} \subseteq V_{n-1}$.  Any vector $v = \sum c_je_j$ with nonzero $c_n$ generates a basis $\{v, Xv, \ldots, X^{n-1}v\}$ for the vector space $V$.  Such a vector cannot lie in $V_{n-1}$, since this $n-1$-dimensional subspace is closed under application of $X$.  Therefore $V_{n-1} = \langle e_1, \ldots, e_{n-1} \rangle = E_{n-1}$.  Conversely, every flag with $V_{n-1} = E_{n-1}$ satisfies $XV_{n-1} \subseteq V_{n-1}$.  It follows that $\Hess(X,h)$ defined by the condition $V_{n-1} = E_{n-1}$, which is the condition defining $\Omega_v$.
\end{proof}

\begin{lemma} \label{lemma:ss n-1}
Let $g$ be the Hessenberg function defined by $g(1)=1$ and $g(i)=n$ for all $i \neq 1$.  Let $h$ be the Hessenberg function defined by $h(i)=n-1$ for all $i \neq n$ and $h(n)=n$.  Let $X$ be a regular semisimple matrix.  Then $\Hess(X,g)$ (respectively $\Hess(X,h)$) consists of $n$ connected components, each of which is isomorphic to $\Fl(n-1)$.
\end{lemma}

\begin{proof}
As before, the Hessenberg variety $\Hess(X,g)$ is defined by the condition that $X V_1 \subseteq V_1$, or equivalently, that $V_1$ is spanned by an eigenvector for $X$.  Since $X$ is regular semisimple, with distinct eigenvalues $\alpha_1,\ldots,\alpha_n$, $V$ breaks up into a direct sum of one-dimensional eigenspaces $L_{\alpha_i}$.  For each choice $V_1 = L_{\alpha_i}$, the rest of the flag can be any of the full flags in $V/V_1$.  Since there are $n$ choices for $V_1$, $\Hess(X,g)$ consists of $n$ disjoint copies of $\Fl(n-1)$, each of which is a translate of $\Omega_u$ inside $\Fl(n)$.

The proof for $\Hess(X,h)$ is analogous.
\end{proof}

\begin{theorem} \label{thm:nilp scheme}
If $X$ is regular nilpotent and $h$ is the Hessenberg function corresponding to the row $(k)$ then the Hessenberg scheme $\HHess(X,h)$ is reduced when $k < n-1$.  When $k=n-1$, the scheme is supported on $\Omega_u$ and satisfies
\begin{eqnarray*}
[\HHess(X,h)] = n [\Omega_u],
\end{eqnarray*}
where $u = s_{n-1}\cdots s_{2} s_{1}$.
\end{theorem}

\begin{proof}
Since $\HHess(X,h)$ has expected codimension, it is Cohen--Macaulay (see \S\ref{subsec:degeneracy}); therefore this scheme is reduced if and only if it is generically reduced.  Since $\HHess(X,h)$ is irreducible, it is generically reduced if the gcd of the coefficients $m^{(n)}_{(k),u}$ is $1$.  When $k<n-1$, Lemma \ref{no $s_n$} shows that there is a permutation $u$ with $m^{(n)}_{(k),u}=1$.  When $k=n-1$, Lemmas \ref{lemma:ss n-1} and \ref{lemma:nilp n-1} show that the regular semisimple Hessenberg variety consists of $n$ copies of $\Fl(n-1)$ and that the regular nilpotent Hessenberg variety consists of one copy of $\Omega_u \isom \Fl(n-1)$.  Since the regular nilpotent Hessenberg scheme $\HHess(X,h)$ is a flat limit of the regular semisimple Hessenberg variety whose reduced variety is $\Omega_u$, the claim holds.
\end{proof}

\appendix
%

\section{Connectivity} \label{sec:connected}

This section proves that regular semisimple Hessenberg varieties are connected if and only if $h(i) \geq i+1$ for all $i \leq n-1$.  In fact, the proof holds for general Lie types.  This is a small corollary of \cite[Corollary 9.(i)]{mps} that we present for completeness.

We use the notation of \cite{mps}, where $G$ is a semisimple algebraic group over $\C$ with fixed Borel subgroup $B$, maximal torus $T \subseteq B$, and Weyl group $W$.  The Lie algebras of $G$, $B$, and $T$ are denoted $\mathfrak{g}$, $\mathfrak{b}$, and $\mathfrak{t}$ respectively.  Let $\Phi^+$ denote the positive roots corresponding to these choices, let $\Phi^-$ be the negative roots, and let $\Delta$ be the simple roots.

A Hessenberg space $H$ is defined to be a (vector) subspace of $\mathfrak{g}$ such that $\mathfrak{b} \subseteq H$ and $H$ is a $\mathfrak{b}$-submodule of $\mathfrak{g}$ (with respect to the Lie bracket).  (When $G = SL_n(\C)$, this agrees with the space defined in Remark \ref{remark:conjugacy}.)
\cite[Lemma 1]{mps} demonstrates a bijection between Hessenberg spaces $H$ and subsets $M \subseteq \Phi^-$ that satisfy the closure condition that if $\alpha \in M$, $\alpha_j \in \Delta$, and $\alpha+\alpha_j \in \Phi^-$, then $\alpha+\alpha_j \in M$.

Given a Hessenberg space $H$ and an element $X \in \mathfrak{g}$, the Hessenberg variety $\Hess(X,H)$ is defined by
\[\Hess(X,H) = \{gB/B \in G/B: g^{-1}Xg \in H\}.\]

\begin{proposition}
If $X$ is a regular semisimple element of $\mathfrak{g}$ then $\Hess(X,H)$ is connected and irreducible if and only if $-\Delta \subseteq M$.
\end{proposition}

\begin{proof}
This follows from \cite[Corollary 9.(i)]{mps}, which states that for regular semisimple $X$, the Hessenberg variety $\Hess(X,H)$ is connected if and only if for each $w \in W$ except $w = e$, the set $w(M) \not \subseteq \Phi^-$.  (Irreducibility comes for free, since regular semisimple Hessenberg varieties are always nonsingular.)

Indeed, suppose $-\Delta \not \subseteq M$ and let $\alpha_j$ be a simple root with $-\alpha_j \not \in M$.  Then the simple reflection $s_j$ corresponding to $\alpha_j$ satisfies the properties that $s_j(\alpha_j) \in \Phi^-$ and that for every other positive root $\alpha$, the image $s_j(\alpha) \in \Phi^+$.  Hence $s_j(M) \subseteq \Phi^-$ and so $\Hess(X,H)$ is not connected.

Conversely, suppose $-\Delta \subseteq M$.  If $w \neq e$ is an element of $W$ then $w(\Phi^-) \not \subseteq \Phi^-$.  Let $\alpha$ be a negative root such that $w(\alpha)$ is positive, and write $\alpha$ in terms of the simple roots as $\alpha = \sum - c_j \alpha_j$, where each $c_j$ is a nonnegative integer.  The reflection $w$ is linear in the roots so $w(\alpha) = \sum c_j w(-\alpha_j)$.  This quantity is positive, so we conclude that at least one of the $w(-\alpha_j)$ is positive.  It follows that $w(M) \not \subseteq \Phi^-$ and so $\Hess(X,H)$ is connected.
\end{proof}


\end{document}